\documentclass[12pt]{amsart}

\usepackage[utf8]{inputenc}
\usepackage{tikz, calc}
\usepackage[justification=centering]{caption}
\usepackage{subcaption}
\usepackage{amssymb, amsthm, fullpage}
\usepackage{graphicx}
\usepackage{verbatim}

\usepackage{algorithm}
\usepackage{algpseudocode}
\usepackage{soul}
\usepackage{url}
\usepackage{longtable}

\newtheorem{definition}{Definition}

\newtheorem{theorem}{Theorem}
\newtheorem{proposition}{Proposition}

\newcounter{x}
\newcounter{y}
\newcounter{x2}
\newcounter{y2}

\newcommand{\vertex}[4][black,]{
    \filldraw[#1] (#2, #3) circle (3pt) node[anchor=west]{#4};
}
\newcommand{\tikzbox}[4][thick,-]{
    \setcounter{x}{#2}
    \setcounter{y}{#3}
    \setcounter{x2}{#2 + #4}
    \setcounter{y2}{#3 + #4}
    
    \draw[#1] (\arabic{x},\arabic{y}) -- (\arabic{x2},\arabic{y2});
    
    \setcounter{y}{#3 + #4 + #4}
    \draw[#1] (\arabic{x2},\arabic{y2}) -- (\arabic{x},\arabic{y});
    
    \setcounter{x2}{#2 - #4}
    \draw[#1] (\arabic{x},\arabic{y}) -- (\arabic{x2},\arabic{y2});
    
    \setcounter{y}{#3}
    \draw[#1] (\arabic{x2},\arabic{y2}) -- (\arabic{x},\arabic{y});
}
\newcommand{\grid}[3][step=1cm, gray, very thin,]{
    \draw[#1] (-0.9,-0.9) grid (#2 + 0.9, #3 + 0.9);
}
\newcommand{\tower}[5][thick, black]{
    \tikzbox{#2}{#3-#4}{#4};
    \vertex{#2}{#3}{#5};
}
\newcommand{\T}{
    \mathbb{T}
}

\begin{document}

% TITLE, AUTHORS, ETC
\title{Computing upper bounds for optimal density \\of $(t,r)$ broadcasts on the infinite grid}

\author{Benjamin F. Drews}
\address{Department of Mathematics and Statistics, Williams College, United States}
\email{bfd2@williams.edu}
\thanks{}

\author{Pamela E. Harris}
\address{Department of Mathematics and Statistics, Williams College, United States}
\email{peh2@williams.edu}
\thanks{P.\,E. Harris was supported by NSF award DMS-1620202.}

\author{Timothy W. Randolph}
\address{Department of Mathematics and Statistics, Williams College, United States}
\email{twr2@williams.edu}
\thanks{}

\date{\today}

\maketitle
%%%%%%%%%%%%
% ABSTRACT %
%%%%%%%%%%%%
\begin{abstract}
The domination number of a finite graph $G$ with vertex set $V$ is the cardinality of the smallest set $S\subseteq V$ such that for every vertex $v\in V$ either $v\in S$ or $v$ is adjacent to a vertex in $S$. A set $S$ satisfying these conditions is called a \emph{dominating set}. In 2015 Blessing, Insko, Johnson, and Mauretour introduced \emph{$(t,r)$ broadcast domination}, a generalization of graph domination parameterized by the nonnegative integers $t$ and $r$. In this setting, we say that the \emph{signal} a vertex $v\in V$ receives from a tower of strength $t$ located at vertex $T$ is defined by $sig(v,T)=max(t-dist(v,T),0)$. Then a \emph{$(t,r)$ broadcast dominating set} on $G$ is a set $S\subseteq V$ such that the sum of all signal received at each vertex $v \in V$ is at least $r$.  In this paper, we consider $(t,r)$ broadcasts of the infinite grid and present a Python program to compute upper bounds on the minimal density of a $(t,r)$ broadcast on the infinite grid. These upper bounds allow us to construct counterexamples to a conjecture by Blessing et al. that the $(t,r)$ and $(t+1, r+2)$ broadcasts are equal whenever $t,r\geq 1$.
\end{abstract}

%%%%%%%%%%%%%%%%%%%%%%%%%%%%%%%%%%%%
% INTRODUCTION                     %
%%%%%%%%%%%%%%%%%%%%%%%%%%%%%%%%%%%%
\section{Introduction}

Let $G$ be a finite graph with vertex set $V$ and edge set $E$ on which the distance between two vertices $u$ and $v$ in $V$, denoted $dist(u,v)$, is defined as the length of the shortest path in $G$ between $u$ and $v$. A set $S\subseteq V$ is called a dominating set of $G$ if for any vertex $v\in V$ either $v\in S$ or $d(u,v)=1$ for some $u\in S$. The cardinality of a smallest dominating set is called the domination number of $G$ and is denoted $\delta(G)$.

We let $G_{m,n}$ denote the \emph{finite grid graph} of dimension $m\times n$. More precisely,
\begin{align*}
G_{m,n} &= (V,E) \ with \\
V &= \{v_{i, j} \ :\ 1\leq i\leq n, 1\leq j\leq m\} \\
E &= \{(v_{i,j}, v_{i+1,j}), (v_{i,j}, v_{i,j+1}) \ :\ 1 \leq i < m, 1 \leq j < n \}.
\end{align*}
Determining the domination number of finite graphs, in particular grid graphs, has received much attention in the graph theory literature; for an overview of the subject, refer to \cite{haynes1998fundamentals}. The 2011 work of Gon{\c{c}}alves, Pinlou, Rao, and Thomass{\'e} \cite{gonccalves2011domination} finally confirmed in the affirmative Chang's 1992 conjecture \cite{chang1992domination} that for every $16 \leq n \leq m$, 
$ \delta(G_{m,n}) = \left \lceil{\frac{(n+2)(m+2)}{5}}\right \rceil$,
thereby establishing the domination number for all finite grid graphs.

Other work in this area expands domination theory by considering variations on domination and includes work on $r$-domination and broadcast domination \cite{dunbar2006broadcasts, griggs1992r}. In 2015, Blessing, Insko, Johnson, and Mauretour introduced $(t,r)$ broadcast domination, another generalization of graph domination defined by the nonnegative integral parameters $r$ and $t$ \cite{blessing2015t}. In this setting, we say that the signal a vertex $v\in V$ receives from a tower of strength $t$ located at vertex $T$ is defined by $sig(v,T)=max(t-dist(v,T),0)$. A $(t,r)$ broadcast dominating set on $G$ is a set $S\subseteq V$ such that the sum of all signal received at each vertex $v \in V$ is at least $r$. The cardinality of the smallest $(t,r)$ broadcast on a finite graph $G$ is called the $(t,r)$ broadcast domination number of $G$. Note that the $(2,1)$ broadcast domination number of a graph $G$ is exactly the classical domination number of $G$.

In this paper, we consider $(t,r)$ broadcasts on the integer lattice graph $G_{\infty}=\mathbb{Z}\times\mathbb{Z}$, which we refer to as the infinite grid. Since the cardinality of any $(t,r)$ broadcast $\T$ on $G_\infty$ will be infinite, we instead compute the \emph{density} of a $(t,r)$ broadcast, which is intuitively defined as the proportion of the vertices of $G_{\infty}$ contained in a $(t,r)$ broadcast $\T$ \cite{DHR}.

\begin{definition}
Given a $(t,r)$ broadcast $\T$ on $G_\infty$, consider the vertex set $V$ of the subgraph $G_{2n+1, 2n+1}$ with its central vertex located at $(0, 0)$. Then the broadcast \emph{density} of a $(t,r)$ broadcast on $G_\infty$ is defined as $$\lim_{n \to \infty} \frac{|\T \cap V|}{|V|}.$$
\end{definition}
The \emph{optimal density} of a $(t,r)$ broadcast on $G_\infty$, denoted $\delta_{t,r}(G_{\infty})$, is the minimum broadcast density over all $(t,r)$ broadcasts and we say that a $(t,r)$ broadcast is \emph{optimal} if its density is optimal.
In previous work, the authors have established the optimal densities of $(t,r)$ broadcasts for all $t\geq 1$ and  $r=1,2$ \cite{DHR}.

\begin{theorem}[Theorems 1, 2, 3 in \cite{DHR}]\label{thm:main} If $t$, $r \in \mathbb{Z}^+$, then
\begin{itemize}
\item $\delta_{t,1}(G_{\infty}) = \frac{1}{2t^2 - 2t + 1}$
\item $\delta_{t,2}(G_{\infty}) = \begin{cases}\frac{1}{3}&\mbox{if $t=2$}\\\frac{1}{2(t-1)^2}&\mbox{if $t>2$}\end{cases}$
\item $\delta_{t,3}(G_\infty) \ \leq \ \delta_{t-1,1}(G_\infty)$.
\end{itemize}
\end{theorem}

Unfortunately, the methods employed in \cite{DHR} to establish statements 1 and 2 of Theorem \ref{thm:main} do not easily extend to compute optimal $(t,r)$ broadcast densities for $r\geq 3$. Hence, a more computational approach is necessary and motivates this work. In this paper, we present a Python program\footnote{The program is available for download from GitHub \cite{DHRprogram}.} to compute upper bounds for the optimal density of $(t,r)$ broadcast domination of $G_\infty$ for any $t>r\geq 1$. 

To compute an upper bound on the $(t,r)$ broadcast density for given $t$ and $r$, our program systematically checks sets of vertices of known broadcast densities and returns the $(t,r)$ broadcast with the lowest density among them. The sets of vertices used by our program are called the \emph{standard patterns}, which we define as follows.

\begin{definition}
The \emph{standard pattern} is defined by the positive integers $d$ and $e$ as
\[p(d,e) = \{(dx+ey, y) : x,y \in \mathbb{Z}\}.\]
\end{definition}

\begin{figure}[h!]
    \begin{center}

    \begin{subfigure}{0.5\textwidth}
    \centering
    \begin{tikzpicture}[scale=0.5]
            \clip (-0.3, -0.3) rectangle (14.3, 8.3);
            \grid[lightgray]{14}{8}
            \vertex{0}{0}{}
            \vertex{4}{0}{}
            \vertex{8}{0}{}
            \vertex{12}{0}{}
            \vertex{2}{1}{}
            \vertex{6}{1}{}
            \vertex{10}{1}{}
            \vertex{14}{1}{}
            \vertex{0}{2}{}
            \vertex{4}{2}{}
            \vertex{8}{2}{}
            \vertex{12}{2}{}
            \vertex{2}{3}{}
            \vertex{6}{3}{}
            \vertex{10}{3}{}
            \vertex{14}{3}{}
            \vertex{0}{4}{}
            \vertex{4}{4}{}
            \vertex{8}{4}{}
            \vertex{12}{4}{}
            \vertex{2}{5}{}
            \vertex{6}{5}{}
            \vertex{10}{5}{}
            \vertex{14}{5}{}
            \vertex{0}{6}{}
            \vertex{4}{6}{}
            \vertex{8}{6}{}
            \vertex{12}{6}{}
            \vertex{2}{7}{}
            \vertex{6}{7}{}
            \vertex{10}{7}{}
            \vertex{14}{7}{}            
            \vertex{0}{8}{}
            \vertex{4}{8}{}
            \vertex{8}{8}{}
            \vertex{12}{8}{}
    \end{tikzpicture}
    \caption{ The standard pattern $p(4, 2)$. }
    \end{subfigure}%
    \begin{subfigure}{0.5\textwidth}
    \centering
    \begin{tikzpicture}[scale=0.5]
            \clip (-0.3, -0.3) rectangle (14.3, 8.3);
            \grid[lightgray]{14}{8}
            \vertex{3}{0}{}
            \vertex{8}{1}{}
            \vertex{0}{2}{}
            \vertex{13}{2}{}
            \vertex{5}{3}{}
            \vertex{10}{4}{}
            \vertex{2}{5}{}
            \vertex{7}{6}{}
            \vertex{12}{7}{}
            \vertex{4}{8}{}
            \vertex{17}{8}{}
            \vertex{9}{9}{}
            \vertex{1}{10}{}
            \vertex{14}{10}{}
    \end{tikzpicture}
    \caption { The standard pattern $p(13, 5)$. }
    \end{subfigure}%
    \end{center}
    \caption{Examples of standard patterns on the infinite grid.}
    \label{(3,1)-optimal}
\end{figure}
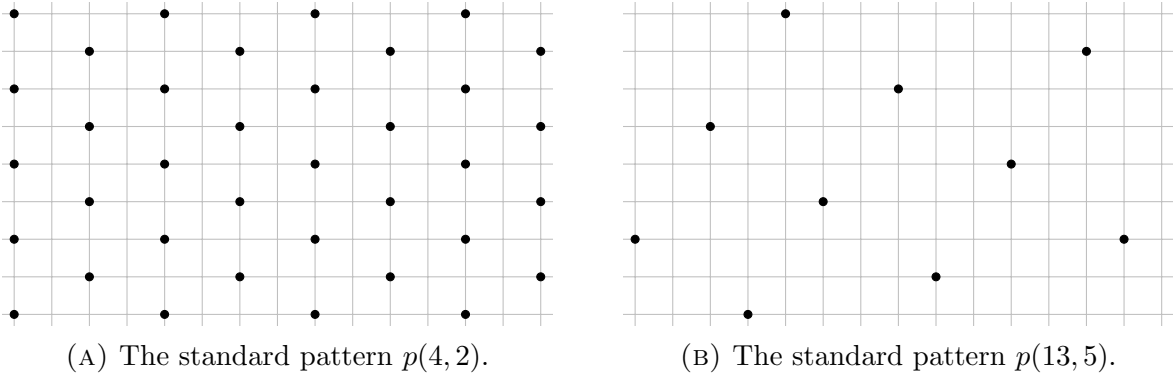

The standard patterns $p(4,2)$ and $p(13,5)$ are depicted in Figure \ref{(3,1)-optimal}. Since, for any positive integer $d$, the standard pattern $p(d,e)$ uses one out of every $d$ vertices on each horizontal line of the grid, it follows that the density of the standard pattern $p(d,e)$ is  $\frac{1}{d}$. We use the notation $\T(d,e)$ to refer to the set of vertices on $G_\infty$ corresponding to the standard pattern $p(d,e)$. Note that $\T(d,e)$ is not necessarily a $(t,r)$ broadcast for any values of $t$ and $r$. When $\T(d,e)$ is in fact a broadcast we call it a \emph{standard broadcast}.

The standard patterns are convenient for computing upper bounds for the optimal density of $(t,r)$ broadcasts because of their regular structure. However, not all standard patterns are $(t,r)$ broadcasts for a given $t$ and $r$. To determine if $\T(d,e)$ is a $(t,r)$ broadcast, we need only to ensure that the sum of signal received is at least $r$ at $d$ specific vertices.  To make this observation precise, we prove the following proposition.

\begin{proposition} 
Let $t$ and $r$ be positive integers. Then $\T(d,e)$ is a $(t,r)$ broadcast if and only if the sum of all signal is at least $r$ for every vertex in the set $\{v_{i,0} \ : \ 0\leq i<d\}$.
\label{prop_stdbroadcast}
\end{proposition}

\begin{proof}
The sufficiency condition follows directly from the definition of a $(t,r)$ broadcast.

To establish the necessity condition suppose that under the set of broadcast towers $\T(d,e)$ the sum of all signal is at least $r$ for every vertex in the set $\{v_{i,0} \ : \ 0\leq i<d\}$, and let $v_{x',y}$ be a vertex of $G_\infty$. By the definition of a standard pattern, there exist integers $x$ and $k$, $0\leq k<d$, such that \[x'=dx+ey+k.\]
Thus $v_{x',y}$ = $v_{dx+ey+k,y}$ for some integers $x$ and $k$. For every tower $T_0$ located at the point $(p,q)$, by the definition of a standard broadcast, there exists a tower $T_1$ located at the point $(dx+ey+p,y+q)$ with $sig(v_{dx+ey+k,y},T_1)=sig(v_{k,0},T_0)$. As the sum of all signal at $v_{k,0}$ is at least $r$, the sum of all signal at $v_{x',y}$ is at least $r$.
\end{proof}

Our program computes an upper bound on the optimal $(t,r)$ broadcast equal to the minimum density over all standard $(t,r)$ broadcasts. While this result is not necessarily optimal, most known optimal $(t,r)$ broadcasts do correspond to standard patterns, including optimal broadcasts for $(t,1)$ with $t\geq 1$, $(t,2)$ with $t\geq 2$, and $(3,3)$ \cite{blessing2015t,DHR}. The $t=2$, $r=3$ case is the only known $(t,r)$ pair for which there exists an optimal non-standard $(t,r)$ broadcast and no optimal standard $(t,r)$-broadcast \cite{DHR}.

We end this section by remarking that the third statement of Theorem \ref{thm:main} provides further evidence for the conjecture of Blessing et al. that the optimal $(t,1)$ and $(t+1,3)$ broadcasts are identical for all $t\geq 3$. However, their conjecture stated more broadly that the optimal $(t,r)$ and $(t+1,r+2)$ broadcasts were identical for all $t\geq 3$, which is false in general. As an application of our program, we present counterexamples to this conjecture in Section \ref{sec:ex}.

\section{Computer implementation}\label{comp}

\subsection{Implementation Details}

The program takes as input positive integers $t$ and $r$ and computes the density of the optimal standard $(t,r)$ broadcast as detailed in Algorithm \ref{alg}.

\begin{algorithm}
    \caption{Optimal Standard $(t,r)$ Broadcast}
    \label{alg}
    \begin{algorithmic}[1]
    \Procedure{MinDensity}{$t$,$r$}
        \State $d_{max}$ = \sc{DMax}($t$, $r$)
        \State $d_{best} \gets 0$
        \State $d \gets 1$
        \While{ $d\leq d_{max}$ }
            \State $e \gets 0$
            \While{ $e < d$ }
                \If{\sc{IsBroadcast}$(t,r,d,e)$}
                    \State $d_{best} \gets d$
                \EndIf
                \State $e \gets e+1$
            \EndWhile
            \State $d \gets d+1$
        \EndWhile \\
        \Return $\frac{1}{d_{best}}$
    \EndProcedure
    \end{algorithmic}
\end{algorithm}

To compute $d_{\max}$, we first compute the \emph{usable signal} emitted by a broadcast tower of strength $t$ on the infinite grid $G_\infty = (V_\infty, E_\infty)$, which is given by the equation
\[\sum_{v \in V_\infty}\min(sig(v,T), r).\]
This value represents the total signal generated by the tower for all nearby vertices, not including the unnecessary signal provided to each vertex $v$ with $sig(v,T) > r$. We divide the usable signal by $r$ to get a conservative lower bound $\delta_{\min}(t,r)$ on the optimal $(t,r)$ broadcast density. Hence, the maximum distance between tower vertices on the same horizontal line in a standard $(t,r)$ broadcast, is equal to
\[d_{\max} \ = \ \left\lfloor \frac{1}{\delta_{\min}(t,r)} \right\rfloor. \]

For each $d \in \{1,2, ... \ d_{\max}\}$, the program iterates through each $e \in \{0,1, \ldots,  \ d-1\}$ and checks to see if $\T(d,e)$ is a standard $(t,r)$ broadcast. The program then returns the optimal standard broadcast and its density.

To check if a set of vertices $\T(d,e)$ is a $(t,r)$ broadcast, the program creates a grid window $W$ consisting of every vertex within distance $t$ of any vertex $v \in \{v_{i,0} \ : \ 0\leq i<d\}$. Thus every tower in $\T(d,e)$ that contributes signal to a vertex in $\{v_{i,0} \ : \ 0\leq i<d\}$ is included in the window. The program then places towers at each vertex $T \in W\cap\T(d,e)$ and computes the total signal of each vertex $v \in \{v_{i,0} \ : \ 0\leq i<d\}$. By Proposition \ref{prop_stdbroadcast}, if each vertex in this set has total signal at least $r$, then $\T(d,e)$ is a standard $(t,r)$ broadcast.

\subsection{Computations}\label{sec:ex}
Blessing et al. prove that the optimal $(3,3)$ and $(2,1)$ broadcast densities are equal on large grids \cite{blessing2015t}, and further conjecture that the optimal $(t,r)$ and $(t+1,r+2)$ broadcast densities are equal for all $(t,r)$. However, with the exception of the $r=1$ case, the conjecture is broadly false.

\renewcommand{\arraystretch}{1.5}
{\setlength\tabcolsep{4pt} % default value: 6pt
\begin{footnotesize}
 \begin{centering}
 \begin{longtable}
 {||c | l  l  l  l  l  l  l  l||} 
 \hline
 $t \setminus r$ & 1 & 2 & 3 & 4 & 5 & 6 & 7 & 8 \\%& 9 & 10\\
 \hline\hline
 1 & $\T$(1,0) & N/A             & N/A         & N/A          & N/A         & N/A         & N/A         & N/A\\%&N/A         & N/A\\
 \hline
 2 & $\T$(5,2) & $\T$(3,1)       & $\T$(1,0)   & $\T$(1,0)    & $\T$(1,0)   & $\T$(1,0)   & N/A         & N/A         \\%&N/A         & N/A \\
 \hline
 3 & $\T$(13,5) & $\T$(8,3)      & $\T$(5,1)   & $\T$(4,1)    & $\T$(3,0)   & $\T$(2,0)   & $\T$(2,0)   & $\T$(2,0)   \\%&$\T$(1,0)   &$\T$(1,0)\\
 \hline
 4 & $\T$(25,7) & $\T$(18,5)     & $\T$(13,5)  & $\T$(10,3)   & $\T$(8,2)   & $\T$(7,2)   & $\T$(5,1)   & $\T$(5,1)   \\%&$\T$(4,2)   &$\T$(4,2)\\
 \hline
 5 & $\T$(41,9) & $\T$(32,7)     & $\T$(25,7)  & $\T$(18,4)   & $\T$(14,4)  & $\T$(13,3)  & $\T$(11,2)  & $\T$(10,2)  \\%&$\T$(9,2)   &$\T$(8,3)\\
 \hline
 6 & $\T$(61,11) & $\T$(50,9)    & $\T$(41,9)  & $\T$(34,13)  & $\T$(26,10) & $\T$(22,5)  & $\T$(19,7)  & $\T$(17,4)  \\%&$\T$(15,3)  &$\T$(14,3)\\
  \hline 
 7 & $\T$(85,13) & $\T$(72,11)   & $\T$(61,11) & $\T$(50,9)   & $\T$(42,16) & $\T$(36,15) & $\T$(29,12) & $\T$(26,5)  \\%&$\T$(23,7)  &$\T$(21,4)\\
 \hline
 8 & $\T$(113,15) & $\T$(98,13)  & $\T$(85,13) & $\T$(74,31)  & $\T$(62,26) & $\T$(54,15) & $\T$(43,12) & $\T$(39,16) \\%&$\T$(34,8)  &$\T$(32,14)\\
 \hline
 9 & $\T$(145,17) & $\T$(128,15) & $\T$(113,15)& $\T$(98,13)  & $\T$(86,36) & $\T$(76,21) & $\T$(65,18) & $\T$(58,17) \\%&$\T$(49,18) &$\T$(46,17)\\
  \hline
 10 & $\T$(181,19)& $\T$(162,17) & $\T$(145,17)& $\T$(130,57) & $\T$(114,50)&$\T$(102,39)& $\T$(89,34) & $\T$(78,17) \\%&$\T$(68,20) &$\T$(62,23)\\
  \hline
 11 & $\T$(221,21)& $\T$(200,19) & $\T$(181,19)& $\T$(162,17) & $\T$(146,64)&$\T$(132,39)& $\T$(115,34)& $\T$(106,23)\\%&$\T$(92,20) &$\T$(84,19)\\
  \hline
 12 & $\T$(265,23)& $\T$(242,21) & $\T$(221,21) & $\T$(202,91)& $\T$(182,82)&$\T$(166,49)& $\T$(149,44)& $\T$(134,29)\\%&$\T$(120,26)&$\T$(110,25)\\
  \hline
 13 & $\T$(313,25)& $\T$(288,23) & $\T$(265,23) & $\T$(242,21)& $\T$(222,100)&$\T$(204,75)& $\T$(185,68)& $\T$(170,47)\\%&$\T$(152,42)&$\T$(140,25)\\
  \hline
 14 & $\T$(365,27)& $\T$(338,25) & $\T$(313,25) & $\T$(290,133)& $\T$(266,122)&$\T$(246,75) & $\T$(223,68)& $\T$(206,47)\\%&$\T$(188,52)&$\T$(174,31)\\
  \hline
 15 & $\T$(421,29)& $\T$(392,27) & $\T$(365,27) & $\T$(338,25)& $\T$(314,144)&$\T$(292,89) & $\T$(269,82)& $\T$(250,57)\\%&$\T$(228,52)&$\T$(212,81)\\
  \hline
\caption{Best standard $(t,r)$ broadcasts for $1\leq t\leq 15$ and $1\leq r\leq 8$.}
\label{bigtable}
\end{longtable}
\end{centering}
\end{footnotesize}
}
Table \ref{bigtable} lists the best standard $(t,r)$ broadcasts for all $1 \leq r \leq 8$ and $1 \leq t \leq 15$, as computed by our program. If multiple standard broadcasts achieve the best broadcast density, the broadcast $\T(d,e)$ with the lowest offset value $e$ is listed. For instance, although $\T(5,1)$, $\T(5,2)$, $\T(5,3)$, and $\T(5,4)$ are optimal $(3,3)$ broadcasts, only $\T(5,1)$ is listed. Recall that the density of the standard $(t,r)$ broadcast $\T(d,e)$ is $\frac{1}{d}$.

Note that the first two statements of Theorem \ref{thm:main} confirm that the standard $(t,1)$ and $(t,2)$ broadcasts listed in the table are optimal, as proved in \cite{DHR}. The data  provide further evidence that the conjecture of Blessing et al. is true in some cases. First, the optimal $(t,1)$ broadcasts are best standard $(t+1,3)$ broadcasts, providing support for the limited conjecture that the optimal $(t,r)$ and $(t+1,r+2)$ broadcast densities are equal in the $r=1$ case \cite{DHR}. Furthermore, the optimal $(t,2)$ broadcasts appear to be equal in density to the best standard $(t+1,4)$ broadcasts when $t$ is even and greater than $3$, indicating a second case in which the conjecture of Blessing et al. might hold.

However, Table \ref{bigtable} provides a variety of counterexamples that indicate the conjecture is false. Although the conjecture would entail that the optimal $(t,1)$ and $(t+2,5)$ broadcast densities are equal, the densities of the optimal $(t,1)$ broadcasts are in fact consistently greater than the corresponding densities of the best standard $(t+2,5)$ broadcasts. Figure \ref{fig:conj_false} displays an optimal $(2,1)$ broadcast and a $(4,5)$ broadcast with lower broadcast density. Additionally, the densities of the optimal $(t,2)$ broadcasts are often greater than those of the best standard $(t+1,4)$ broadcasts and consistently greater than those of the best standard $(t+2,6)$ broadcasts. Thus, establishing that the conjecture of Blessing et al. is false in the general case, the best standard broadcasts provide tighter upper bounds on the optimal $(t,r)$ broadcasts. 

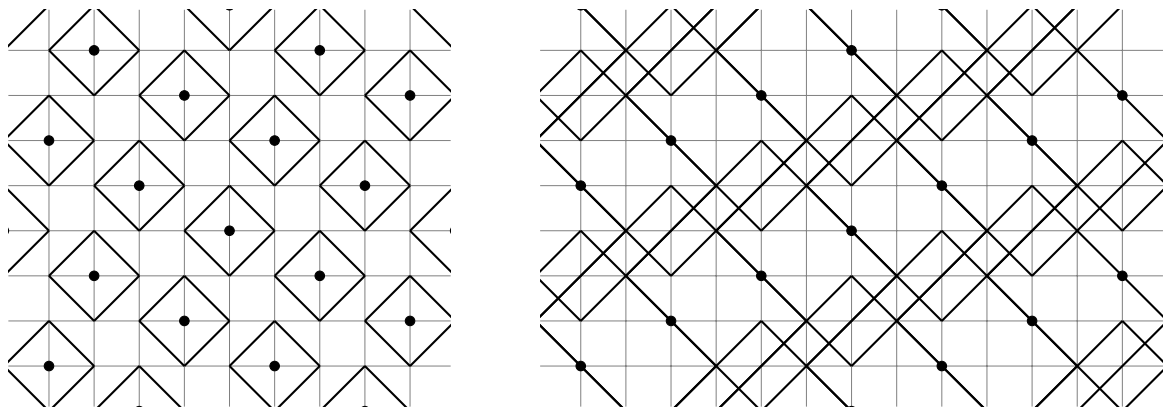
\begin{figure}[h!]
    \begin{subfigure}{0.5\textwidth}
    \centering
    \begin{tikzpicture}[scale=0.6]
        \clip (-0.9, -0.9) rectangle (8 + 0.9, 7 + 0.9);
        \grid{8}{7}
        \tower{0}{0}{1}{}
        \tower{5}{0}{1}{}
        \tower{10}{0}{1}{}
        
        \tower{3}{1}{1}{}
        \tower{8}{1}{1}{}
        \tower{3}{6}{1}{}
        \tower{8}{6}{1}{}
        
        \tower{1}{2}{1}{}
        \tower{6}{2}{1}{}
        \tower{1}{7}{1}{}
        \tower{6}{7}{1}{}
        
        \tower{-1}{3}{1}{}
        \tower{4}{3}{1}{}
        \tower{9}{3}{1}{}
        \tower{-1}{8}{1}{}
        \tower{4}{8}{1}{}
        \tower{9}{8}{1}{}
        
        \tower{2}{4}{1}{}
        \tower{7}{4}{1}{}
        \tower{2}{-1}{1}{}
        \tower{7}{-1}{1}{}
        
        \tower{0}{5}{1}{}
        \tower{5}{5}{1}{}
        \tower{10}{5}{1}{}
    \end{tikzpicture}
    \caption{The standard $(2,1)$ broadcast $\T(5,3)$.}
    \end{subfigure}%
    \begin{subfigure}{0.5\textwidth}
    \centering
    \begin{tikzpicture}[scale=0.6]
        \clip (-0.9, -0.9) rectangle (12 + 0.9, 7 + 0.9);
        \grid{12}{7}
        \tower{0}{0}{3}{}
        \tower{8}{0}{3}{}
        \tower{0}{4}{3}{}
        \tower{8}{4}{3}{}
        \tower{0}{8}{3}{}
        \tower{8}{8}{3}{}
        
        \tower{2}{1}{3}{}
        \tower{10}{1}{3}{}
        \tower{2}{5}{3}{}
        \tower{10}{5}{3}{}
        \tower{2}{9}{3}{}
        \tower{10}{9}{3}{}
        
        \tower{4}{-2}{3}{}
        \tower{12}{-2}{3}{}
        \tower{4}{2}{3}{}
        \tower{12}{2}{3}{}
        \tower{4}{6}{3}{}
        \tower{12}{6}{3}{}
        
        \tower{-2}{-1}{3}{}
        \tower{6}{-1}{3}{}
        \tower{-2}{3}{3}{}
        \tower{6}{3}{3}{}
        \tower{-2}{7}{3}{}
        \tower{6}{7}{3}{}
    \end{tikzpicture}
    \caption{The standard $(4,5)$ broadcast $\T(8,2)$.}
    \end{subfigure}
    
    \caption{The standard $(4,5)$ broadcast $\T(8,2)$ has a lower broadcast density than the optimal $(2,1)$ broadcast $\T(5,3)$. }
    \label{fig:conj_false}
\end{figure}

%\printbibliography

\bibliography{main}
\bibliographystyle{plain}

\end{document}